\documentclass{amsart}   	
\usepackage{geometry}               
\usepackage{graphicx}
\usepackage[usenames,dvipsnames]{xcolor}
\usepackage{amssymb,amsmath,amsthm,amsfonts}

\usepackage[shortlabels]{enumitem}

\usepackage{hyperref} 
\usepackage[capitalise, nameinlink, noabbrev]{cleveref} 
\hypersetup{colorlinks=true, 
	linkcolor=blue,
	citecolor=red,
}
\allowdisplaybreaks

\usepackage[english]{babel}
\usepackage[utf8]{inputenc}

\usepackage{mathtools}
\usepackage{esint} 
\usepackage{dsfont}

\usepackage[sc]{mathpazo}

\usepackage{tikz}
\usepackage{xfrac}

\newtheorem{proposition}{Proposition}[section]
\newtheorem{theorem}[proposition]{Theorem}

\newtheorem{lemma}[proposition]{Lemma}

\theoremstyle{definition}
\newtheorem{definition}[proposition]{Definition}
\theoremstyle{remark}
\newtheorem{remark}[proposition]{Remark}
\numberwithin{equation}{section}

\newcommand{\eps}{\varepsilon}

\newcommand{\N}{{\mathbb{N}}}

\newcommand{\R}{{\mathbb{R}}}
\newcommand{\C}{{\mathbb{C}}}

\DeclareMathOperator{\diverg}{div}

\title{On the fine structure of the solutions to nonlinear thin two-membrane problems in 2D}
\author{Lorenzo Ferreri, Luca Spolaor, Bozhidar Velichkov}

\begin{document}
\maketitle

\begin{abstract}
We prove a structure theorem for the solutions of nonlinear thin two-membrane problems in dimension two. Using the theory of quasi-conformal maps, we show that the difference of the sheets is topologically equivalent to a solution of the linear thin obstacle problem, thus inheriting its free boundary structure. More precisely, we show that even in the nonlinear case the branching points can only occur in finite number. We apply our methods to one-phase free boundaries approaching a fixed analytic boundary and to the solutions of a one-sided two-phase Bernoulli problem. 
\end{abstract}

\tableofcontents

\noindent
{\footnotesize \textbf{AMS-Subject Classification}}. 
{\footnotesize 35R35
} 

\noindent {\footnotesize {\bf Acknowledgments.}
	L.F. and B.V. are supported by the European Research Council's (ERC) project n.853404 ERC VaReg - \it Variational approach to the regularity of the free boundaries\rm, financed by the program Horizon 2020.
	L.F. is a member of INDAM-GNAMPA.
B.V. has been partially supported by the MIUR-PRIN Grant 2022R537CS ``$NO^3$'' and by the MIUR Excellence Department Project awarded to the Department of Mathematics, University of Pisa, CUP I57G22000700001. L.S. acknowledges the support of the NSF Career Grant DMS 2044954.}

\section{Introduction}\label{section:intro}

We say that a function $u:B_1^+\to\R$ is a variational solution to a variable coefficients thin obstacle problem if 
\begin{equation}\label{e:variational-inequality-thin}
\int_{B_1^+} A \nabla u \cdot \nabla (u - v) \le 0
\end{equation}
for some uniformly elliptic matrix $A$ with $L^\infty$ coefficients and all functions
\[
v \in H^1(B_1^+), \quad v \ge 0 \text{ on } B_1' \quad \text{and} \quad v = u \text{ on } \partial B_1 \cap \{ x_2 \ge 0 \}.
\]
When $A=I$ is the identity matrix, we say that $u$ is a solution of the harmonic thin obstacle problem. The following is the main result of the paper.

\begin{theorem}[Topological properties of solutions to the thin obstacle problem]\label{t:thin-obstacle-general}
Let $u:B_1^+\to\R$ with $u \in H^1(B_1^+)$ and $u = u_0$ on $\partial B_1 \cap \{ x_2 \ge 0 \}$ be a variational solution to a variable coefficients thin obstacle problem for some uniformly elliptic matrix $A$ with $L^\infty$ coefficients. Then, there exists a quasiconformal map $f:\R^2 \to \R^2$ such that the function
\[
h \coloneqq u \circ f
\]
is a solution of the harmonic thin obstacle problem in a neighborhood of the origin. In particular in the ball neighborhood, the non-contact set of $u$ 
\[
\mathcal C:=\{(x,0)\in B_1'\ :\ u(x,0)>0\}
\]
is composed of a finite number of disjoint open segments.
\end{theorem}
%

In \cite{DePhilippisSpolaorVeluchkov2021:QuasiConformal2D} the authors showed that the gradient of a solution to a nonlinear thin obstacle problem is a quasiconformal map. Here, instead, we show that, up to a quasiconformal transformation of the domain, solutions of the nonlinear problem are equivalent to solutions of the linear problem. This quasiconformal equivalence is more flexible than working at the level of the gradient, and it allows us to deal with \emph{non-zero average situations}, cp. Remark \ref{rem:differences}. In particular we can analyse the fine structure of the singular set of three related problems: a nonlinear thin-two membrane problem, a one phase Bernoulli problem with analytic obstacle and a one-sided two-phase Bernoulli problem, none of which can be dealt with working at the level of the gradient.

\medskip

We say that a couple $u,v\in C^{1,\alpha}(B_1^+\cup B_1')$ is a solution to an analytic nonlinear thin two-membrane problem in $B_1^+$ if
\begin{equation}\label{e:thin-two-membrane-system}
\begin{cases}
    \diverg\left( \nabla \mathcal F\left( \nabla u \right) \right) = 0 & \text{in } B_1^+, \\
    \diverg\left( \nabla\mathcal F\left( \nabla v \right) \right) = 0 & \text{in } B_1^+, \\
    u\ge v & \text{on } B_1', \\
    e_2 \cdot \nabla \mathcal F\left( \nabla u \right) =0 & \text{on } \{ u>v \} \cap B_1',\\
     e_2 \cdot \nabla \mathcal F\left( \nabla v \right)  = 0 & \text{on } \{ u>v \} \cap B_1',\\
     e_2 \cdot \Big(\nabla \mathcal F\left( \nabla u \right)-\nabla \mathcal F\left( \nabla v \right)\Big) \le0 & \text{on } \{ u=v \} \cap B_1'.\\
\end{cases}
\end{equation}
for an analytic function $\mathcal F\colon \R^2\to \R$ satisfying
\begin{equation}\label{e:definition-mathcal-F}
\mathcal F(0)=0\,,\qquad \nabla \mathcal F(0)=0\qquad\text{and}\qquad \nabla^2\mathcal F(0)=I\,. 
\end{equation}
Then we have the following result:

\begin{theorem}\label{t:thin-two-membrane-general}
    Let $u,v \in C^{1,\alpha}(B_1^+\cup B_1')$ be a solution to the analytic nonlinear thin two-membrane problem \eqref{e:thin-two-membrane-system} in $B_1^+$, then in the ball $B_{1/2}$, the non-contact set 
\[
\mathcal C:=\{(x,0)\in B_1'\ :\ u(x,0)>v(x,0)\}
\]
is composed of a finite number of disjoint open segments. 
\end{theorem}
The key observation to pass from \cref{t:thin-obstacle-general} to \cref{t:thin-two-membrane-general} is the fact that the difference $u-v$ is a variational solution to a linear thin obstacle problem with $C^{0, \alpha}$ coefficients, and so it follows in the more general framework of \cref{t:thin-obstacle-general}.

%
%
%
%
%
%
%
%
%
%


\medskip
Next we fix an analytic function $\phi:B_1'\to\R$
and we let 
\[
\mathcal{A} \coloneqq \mathrm{Epi}(\phi) \cap B_1=\{(x',x_d)\in B_1\ :\ x_d>\phi(x')\}\,.
\]
We consider minimizers $u\in H^1(B_1)$ of the one-phase Bernoulli energy
\begin{equation}\label{eqn:onePhaseEnergyU}
J_1(u) \coloneqq \int_{B_1} \vert \nabla u \vert^2\,dx + \vert 
\{ u > 0 \} \cap B_1 \vert \quad 
\end{equation}
under the additional geometric constraint
\begin{equation}\label{eqn:onePhaseConstraint}
\{ u>0 \} \subseteq \mathcal{A} \text{ a.e. in } B_1,
\end{equation}
In the recent paper by Chang-Lara and Savin \cite{ChangLaraSavin:BoundaryRegularityOnePhase} it was shown that, if $\phi \in C^{1, 1/2}$, then in a neighborhood of any contact point $z_0=(z',\phi(z'))\in\partial\Omega_u\cap B_1$ the boundary $\partial\Omega_u$ is a $C^{1,1/2}$ graph over the hyperplane $\{x_d=0\}$. 
More precisely, 
there are a radius $\rho>0$ and a $C^{1,1/2}$ function
	$$f:B_\rho'(z_0)\to\R\ ,\quad f\ge\phi\ \text{ on }\  B_\rho'(z_0),$$
	such that, up to a rotation and translation of the coordinate system, we have 
	\begin{equation}\label{eq:fbgraph}
	\begin{cases}
	\begin{array}{rcl}
	u(x)>0&\quad\text{for}\quad&x\in (x',x_d)\in B_\rho(z_0) \quad\text{such that}\quad x_d>f(x');\\
	u(x)=0&\quad\text{for}\quad&x\in (x',x_d)\in B_\rho(z_0) \quad\text{such that}\quad  x_d\le f(x').
	\end{array}
	\end{cases}
	\end{equation}	
	We denote by $\mathcal C_1(u)$ the contact set of the free boundary $\partial\Omega_u$ with the graph $\{x_d=\phi(x')\}$
	$$
	\mathcal C_1(u):=\{x_d=\phi(x')\}\cap\partial\Omega_u\,,$$
	 and by $\mathcal B_1(u)$ the set of points at which the free boundary separates from $\{x_d=\phi(x')\}$ : 
	 $$\mathcal B_1(u):=\Big\{x\in \mathcal C_1(u)\ :\ B_r(x)\cap  \big(\partial\Omega_u\setminus\{x_d=\phi(x')\}\big)\neq\emptyset\ \text{ for every }\ r>0\Big\}.$$
	By $\mathcal S_{1}(u)$ we denote the set of points in $\mathcal C_1(u)$ at which $u$ has gradient precisely equal to $1$
	\begin{equation}\label{e:def-S-one-phase}
	\mathcal S_{1}(u):=\big\{z\in \mathcal C_1(u)\,:\,|\nabla u|(z)=1\big\}.
	\end{equation}
	We notice that a priori the set $\mathcal C_1(u)$ is no more than a closed subset of $\{x_d=\phi(x')\}$. Moreover, if at a point $x=(x',\phi(x'))$ we have that $|\nabla u|(x',\phi(x'))>1$, then this point is necessarily in the interior of $\mathcal C_1(u)$ in the graph $\{x_d=\phi(x')\}$. Thus, 
	\begin{center}
	$\mathcal S_1(u)$ contains all branch points,  $\mathcal B_1(u)\subset \mathcal S_1(u)$.
\end{center}

	\begin{theorem}[Structure at branch points in the one phase problem with obstacle]\label{thm:onePhaseAnalyticObstacle} Let $u$ be a minimizer of $J_1$ under the constraint \eqref{eqn:onePhaseConstraint} in dimension $d=2$. Then, the following holds:
	\begin{enumerate}[\rm(a)]
	\item\label{item:aa} $\mathcal S_1(u)$ is locally finite and $\mathcal C_1(u)$ is a locally finite union of disjoint closed arcs of graph $\{x_2=\phi(x_1)\}$;\smallskip
	\item\label{item:bb} For every point $z_0\in\mathcal S_1(u)$, one of the following holds:\smallskip
	\begin{enumerate}[\rm(b.1)]
		\item\label{item:bb1} $z_0$ is an isolated point of $\mathcal C_1(u)$ and, in a neighborhood of $z_0$, the free boundary $\partial\Omega_u$ is the graph of an analytic function that vanishes only at $z_0$;\smallskip
		\item\label{item:bb2} $z_0$ lies in the interior of $\mathcal C_1(u)$ and there is $r>0$ such that $u$ is harmonic in $B_r(z_0)$ and $|\nabla u|>1$ at all points of $\{x_2=\phi(x_1)\}\cap B_r(z_0)$ except $z_0$;\smallskip
		\item\label{item:bb3} $z_0$ is an endpoint of a non-trivial interval in the contact set $\mathcal C_1(u)$; moreover, there is an interval $\mathcal I_\rho=(-\rho,\rho)$ and a function $\psi:\mathcal I_\rho\to\R$ continuous at $0$ and such that $\psi(0)>0$ and, up to setting $z_0=0$ and rotating the coordinate axis, 
			\begin{equation}\label{eq:exp}
		f(x)-\phi(x)=
		\begin{cases}
		0 & \text{ if } x\ge  0\\
		x^{2m-\sfrac 12}\,\psi(x) & \text{ if }x<0\,,
		\end{cases}
		\end{equation}
where $f$ is the function from \eqref{eq:fbgraph} and $m\in \N_{\geq 1}$.
	\end{enumerate}
	\end{enumerate}
	\end{theorem}
The key ideas to reduce the proof of this theorem to that of \cref{t:thin-two-membrane-general} are:
\begin{itemize}
    \item following the construction of \cref{s:proof-t2}, we build a function $v:B_1 \to \R$ which is a minimizer of the Alt-Caffarelli functional \cite{AltCaffarelli:OnePhaseFreeBd} and extends the obstacle, in the sense that
    \[
    \partial \{v>0\} \cap B_1 = \partial A \cap B_1,
    \]
    which is possible thanks to the fact that the obstacle is assumed to be analytic;
    \item using the classical partial hodograph transform on the solution $u$ and the extension $v$ separately, we can reduce our problem to the study of solutions to \ref{e:thin-two-membrane-system} for a specific choice of an analytic function $\mathcal F$ satisfying \eqref{e:definition-mathcal-F};
    \item finally, once we know that branching points are isolated, the precise structure of the free-boundary follows from tracing back the change of coordinates.
\end{itemize}

\medskip

Finally, using a similar procedure to the one for \cref{thm:onePhaseAnalyticObstacle}, we deduce the fine structure of the free boundaries of two-phase problems in which the two phases lie on the same side. Precisely, given two positive constants $\Lambda_u \ge \Lambda_v >0$ we denote with $J_2$ the functional
\[
J_2(u, v) \coloneqq \int_{B_1} \vert \nabla u \vert^2 + \vert \nabla v \vert^2 + \Lambda_u \vert \{u>0\} \cap B_1 \vert + \Lambda_v \vert \{v>0\} \cap B_1 \vert.
\]
%
We say that the couple $u,v\in H^1(B_1)$, $u\ge v$, is a minimizer of $J_2$ in $B_1$ if 
$$J_2(u,v)\le J_2(\varphi,\psi)$$
for every couple $\varphi,\psi\in H^1(B_1)$ such that $\varphi\ge \psi$ in $B_1$, $u=\varphi$ on $\partial B_1$ and $v=\psi$ on $\partial B_1$. 

 Given a couple $u,v$ that minimizes $J_2$ in $B_1$, we set $\Omega_u:=\{u>0\}$ and $\Omega_v:=\{v>0\}$ and we notice that $\Omega_v\subset\Omega_u$.  We will call the point on $\partial\Omega_u\cap\partial\Omega_v$ two-phase points and we will say that a two-phase point $x_0\in \partial\Omega_u\cap\partial\Omega_v$ is a branch point if the two free boundaries $\partial\Omega_u$ and $\partial\Omega_v$ are separating, that is, $$|B_r(x_0)\cap(\partial\Omega_u\setminus\partial\Omega_v)|>0\quad\text{for every}\quad r>0;$$ 
we will denote by $\mathcal B_2(u,v)$ the set of branch points. By \cite{FerreriVelichkov2023:oneSidedTwoPhase} we know that in dimension two, the two free boundaries $\partial\Omega_u$ and $\partial\Omega_v$ are graphs of $C^{1,1/2}$ functions, so the branch points are cusp-like singularities on the boundary of $\Omega_u\setminus\Omega_v$. In the next theorem, we prove that the set of branch points is locally finite and we show that the functions $u$ and $v$ are analytic in the neighborhood of a branch point. 

\begin{theorem}[Structure at branch points in a one-sided two-phase problem]\label{t:one-sided-two-phase} In dimension $d=2$, let $u,v\in H^1(B_1)$, $u\ge v$, be minimizers of the functional $J_2$ in $B_1$. Then, the following holds:
	\begin{enumerate}[\rm(a)]
	\item\label{item:aa} the set of branch points $\mathcal B_2(u,v)$ is locally finite;\smallskip
	\item\label{item:bb} for every point $z_0\in\mathcal B_2(u,v)$, one of the following holds:\smallskip
	\begin{enumerate}[\rm(b.1)]
		\item\label{item:bb1} there is a neighborhood $B_r(z_0)$ of $z_0$ such that $\{z_0\}=\partial\Omega_u\cap\partial\Omega_v$ in $B_r(z_0)$;\smallskip
		\item\label{item:bb2}there is a ball $B_r(z_0)$ such that $\Omega_u=\Omega_v$ in $B_r(z_0)$;\smallskip
		
		\item\label{item:bb3} there is an interval $(-\rho,\rho)$ and two  $C^{1,1/2}$ functions $\eta_u,\eta_v:(-\rho,\rho)\to\R$ such that $\eta_u(0)=\eta_v(0)=\eta'_u(0)=\eta'_v(0)$ and such that, up to setting $z_0=0$ and rotating the coordinate axes, the following holds in the square $(-\rho,\rho)\times(-\rho,\rho)$: 
  $$\Omega_u=\{(x,y):y>\eta_u(x)\}\qquad\text{and}\qquad \Omega_v=\{(x,y):y>\eta_v(x)\},$$
  $$\eta_u\equiv\eta_v\quad\text{on}\quad (-\rho,0)\ ;\qquad \eta_u<\eta_v\quad\text{on}\quad (0,\rho)\,;$$
   moreover, there are $m\in\N$, $m\ge 3$, and a function $\psi:\mathcal (-\rho,\rho)\to\R$ continuous at $0$ and such that $\psi(0)>0$ and
			\begin{equation}\label{eq:exp}
		\eta_v(x)-\eta_u(x)=
		\begin{cases}
		0 & \text{ if } x\le  0\\
		x^{2m-\sfrac 12}\,\psi(x) & \text{ if }x>0\,.
		\end{cases}
		\end{equation}
	\end{enumerate}
	\end{enumerate}
	\end{theorem}

\begin{remark}\label{rem:differences}
In the case when the obstacle $\phi$ is $0$, \cref{thm:onePhaseAnalyticObstacle} was proved in  \cite{DePhilippisSpolaorVeluchkov2021:QuasiConformal2D} in two different ways, via a conformal hodograph transform and via a "linear" hodograph transform, but both approaches fail in the presence of a non-trivial $\phi$. \cref{thm:onePhaseAnalyticObstacle} completes the analysis started in \cite{DePhilippisSpolaorVeluchkov2021:QuasiConformal2D} and is also sharp; indeed, by taking a solution of the one-phase problem (without obstacle) one can easily add an artificial $C^\infty$ smooth (but not analytic) obstacle touching the free boundary infinitely many times. Similarly, the conformal hodograph transform method from \cite{DePhilippisSpolaorVeluchkov2021:QuasiConformal2D} does not apply to \cref{t:one-sided-two-phase}. In particular, this leaves an open question regarding the functions $\psi$, responsible for the opening of the cusps in  \cref{thm:onePhaseAnalyticObstacle} and \cref{t:one-sided-two-phase}; we suspect that in both cases these functions are analytic, as it happens in the case of zero obstacle studied in  \cite{DePhilippisSpolaorVeluchkov2021:QuasiConformal2D}, but we couldn't find a simple proof of this fact in the general case. 
\end{remark}

\section{Proof of \cref{t:thin-obstacle-general}}

%
%
%
%
We look for a quasiconformal mapping $f: \C \to \C$ which reduces \eqref{e:variational-inequality-thin} to a harmonic thin obstacle problem. To this aim, let $M:B_1^+ \to \mathcal{M}^{2 \times 2}$ be a symmetric and uniformly elliptic matrix field with $\det M = 1$ a.e. in $B_1^+$ and such that
\[
A \nabla u = M \nabla u \quad \text{a.e. in } B_1^+.
\]
For the existence of such $M$, see e.g. \cite[Chapter 16]{AstalaIwaniecMartin:PDEsQuasiconformal}. Notice that, in general, $M \in L^{\infty}\left( B_1^+ \right)$, but not even continuous. 
Now we require that
\begin{equation}\label{eqn:ThinObstacleQconformalEqn}
\left( Df \right)^T Df = M^{-1} J(z, f) \quad \text{a.e. in } B_1^+,
\end{equation}
where $J(z,f)$ is the Jacobian of the map $f$. Since $\det M = 1$, this is equivalent to the following linear Beltrami equation (e.g. \cite[Chapter 10]{AstalaIwaniecMartin:PDEsQuasiconformal})
%
%
%
\[
\partial_{\overline{z}} f = \mu \partial_z f + \nu \overline{\partial_z f} \quad \text{a.e. in } B_1^+,
\]
for some measurable complex functions $\mu, \nu:B_1^+ \to \C$ satisfying the uniform ellipticity condition
\[
\vert \mu \vert + \vert \nu \vert < 1 \quad \text{a.e. in } B_1^+.
\]
In order to obtain a global Beltrami equation, we extend $\mu$ and $\nu$ to $\tilde{\mu}$ and $\tilde{\nu}$ on the whole complex plane, and we choose the following way:
\[
\tilde{\mu}(z) =
\begin{cases}
    \mu(z) & \text{if } z \in B_1^+ , \\
    \overline{\mu}(\overline{z}) & \text{if } z \in B_1^- ,\\
    0 & \text{otherwise},
\end{cases}
\quad \text{and} \quad
\tilde{\nu}(z) =
\begin{cases}
    \nu(z) & \text{if } z \in B_1^+ , \\
    \overline{\nu}(\overline{z}) & \text{if } z \in B_1^- ,\\
    0 & \text{otherwise}.
\end{cases}
\]
Let $f$ be the unique quasiconformal solution to the Beltrami equation
\[
\partial_{\overline{z}} f = \tilde{\mu} \partial_z f + \tilde{\nu} \overline{\partial_z f} \quad \text{in } \C
\]
normalized so that $f(0) = 0$ and $f(1) = 1$. For the  existence and uniqueness, see e.g. \cite[Chapter 6]{AstalaIwaniecMartin:PDEsQuasiconformal}. Notice that thanks to the uniqueness and the extension chosen for $\tilde{\mu}$ and $\tilde{\nu}$, we also have
\[
f(z) = \overline{f}\left( \overline{z} \right) \text{ for all } z \in \C,
\]
so that $f$ maps the real line onto itself. Now consider the function
\[
w \coloneqq u \circ f^{-1}.
\]
Thanks to \eqref{eqn:ThinObstacleQconformalEqn}, there exists a neighborhood of the origin, which without loss of generality we can assume to be $B_1(0)$, such that $w$ solves the variational thin obstacle problem
\begin{align*}
\int_{B_1^+} \nabla w \cdot \nabla (w - v) &\le 0\quad\text{for every}\quad v \in H^1(B_1^+)\\
&\text{such that}\quad v \ge 0 \text{ on } B_1' \quad \text{and} \quad v = w \text{ on } \partial B_1 \cap \{ x_2 \ge 0 \}.
\end{align*}
Finally, by the result of Lewy \cite{LewyThinObstacle2D}, we get that the contact set is composed of a finite number of disjoint intervals.

\section{Proof of \cref{t:thin-two-membrane-general}}
%
%
We will write an equation for the difference
\[
w \coloneqq u - v.
\]
For any $y=(y_1,y_2)$ and $x=(x_1,x_2)$ and any analytic function $F$ of two variables, we write the k-th Taylor polynomial of $F$ at $y$ as
\[
T_k[x](F(y))=\sum_{i_1=1}^2\sum_{i_2=1}^2\dots\sum_{i_k=1}^2\partial_{i_1i_2\dots i_k}F(y)x_{i_1}x_{i_2}\dots x_{i_k}.
\]
Then, 
\begin{align*}
\partial_1F(x)-\partial_1F(y)&=\sum_{k=1}^{+\infty}\frac{1}{k!}T_k[x-y](\partial_1F(y)) = \\
&=(x_1-y_1) \sum_{k=1}^{+\infty}\frac{1}{k!} T_{k-1}[x-y](\partial_{11}F(y)) + (x_2-y_2) \sum_{k=1}^{+\infty}\frac{1}{k!}T_{k-1}[x-y](\partial_{12}F(y)),\\
\partial_2F(x)-\partial_2F(y)&=\sum_{k=1}^{+\infty}\frac{1}{k!}T_k[x-y](\partial_2F(y)) = \\
&=(x_1-y_1) \sum_{k=1}^{+\infty}\frac{1}{k!} T_{k-1}[x-y](\partial_{21}F(y)) + (x_2-y_2) \sum_{k=1}^{+\infty}\frac{1}{k!}T_{k-1}[x-y](\partial_{22}F(y)).
\end{align*}
In particular, if we denote by $M(x,y)$ the $2\times 2$ symmetric matrix with coefficients 
\begin{equation}\label{eqn:matrixMAnalyticFCoeffs}
    m_{ij}:= \sum_{k=1}^{+\infty}\frac{1}{k!} T_{k-1}[x-y](\partial_{ij}F(y)),
\end{equation}
we have that 
\[
\nabla F(x)-\nabla F(y)=M(x,y)\begin{pmatrix}x_1-y_1\\ x_2-y_2\end{pmatrix}.
\]
This implies that the difference $w$ is the (unique) $C^{1, \alpha}$ viscosity solution of the thin obstacle problem
\begin{equation*}
\begin{cases}
    \diverg\left( M\left( \nabla u, \nabla v \right) \nabla w \right) = 0 & \text{in } B_1^+, \\
    w\ge 0 & \text{on } B_1', \\
    e_2 \cdot M\left( \nabla u, \nabla v \right) \nabla w =0 & \text{on } \{ w>0 \} \cap B_1',\\
    e_2 \cdot  M\left( \nabla u, \nabla v \right) \nabla w \le0 & \text{on } \{ u=v \} \cap B_1',\\
\end{cases}
\end{equation*}
for the symmetric and uniformly elliptic matrix $M$ with $C^{0, \alpha}$ coefficients given by \eqref{eqn:matrixMAnalyticFCoeffs}.

\begin{remark}\label{rmk:Thin2membraneDiffProperty}
We notice that, by the convexity of the problem and symmetry of $M$, it is also the unique solution of 
\begin{equation}\label{eqn:variationalThinObstDefProof}
\min \left\{ \int_{B_1^+} M \nabla w \nabla w : w \in H^1(B_1^+), \, w \ge 0 \text{ on } B_1', \, w = w_0 \text{ on } \partial B_1 \cap \{ x_2 \ge 0 \} \right\} .
\end{equation}
The remaining part of \cref{t:thin-two-membrane-general} follows directly from \cref{t:thin-obstacle-general}.
\end{remark}

\section{Proof of \cref{thm:onePhaseAnalyticObstacle}}\label{sec:proofThmonePhaseAnalyticObstacle}

We divide the proof in two steps: first we prove that the branching points are isolated and then we provide the precise structure of the free-boundary at the branching points.

\subsubsection{Hodographic reduction to a thin two-membrane  problem}\label{subsec:onePhaseClassicalHodograph}
We let $v$ be the smooth solution to the one-phase Bernoulli problem constructed in  \ref{t:2},  with $f=\phi$ the analytic obstacle. Even though the existence of such a function $v$ could probably be deduced from Cauchy-Kovalevskaya theorem, we decided to give a complete proof in the appendix using the techniques in \cite{DePhilippisSpolaorVeluchkov2021:QuasiConformal2D} as we couldn't find a precise reference in the literature. 

next, we introduce the functions $u'$ and $v'$ via the classical hodograph transform, separately for the solution $u$ and the extension $v$. More precisely, they are locally defined via the following relations (which up to a scaling we can assume valid in $B_1^+$).
\begin{align*}
    & \Phi_u(x, y) \coloneqq \left(x, u(x, y)\right) \longleftrightarrow \left(s, u'(s, z)\right) \coloneqq \Phi_u(x, y)^{-1} & (s, z) \in B_1^+, \\
    & \Phi_v(x, y) \coloneqq \left(x, v(x, y)\right) \longleftrightarrow \left(s, v'(s, z)\right) \coloneqq \Phi_v(x, y)^{-1} &  (s, z) \in B_1^+.
\end{align*}
%
%
%
In the following, we will be interested in the linearizations of $u'$ and $v'$. More precisely, we define
\[
\tilde{u} \coloneqq u' - z \quad \text{and} \quad \tilde{v} \coloneqq v' - z
\]
Let us proceed to write down the problem solved by $\tilde{u}$ and $\tilde{v}$. Actually we do the computations only for $\tilde{u}$, since there is no difference in the derivation for $\tilde{v}$.

Let us first write the energy in \eqref{eqn:onePhaseEnergyU} in terms of $u'$. To this aim, differentiating the relation
\[
u\left( s, u'(s, z) \right) = z
\]
we have
\begin{equation}\label{eqn:HodographDerivatives}
    \partial_x u + \partial_y u \partial_s u' = 0 \quad \text{and} \quad \partial_y u \partial_z u' = 1.
\end{equation}
Moreover, from \eqref{eqn:HodographDerivatives}
\begin{equation}\label{eqn:HodographJacobian}
\det D\Phi_u^{-1} = \frac{1}{\det D\Phi_u} = \frac{1}{\partial_y u }  = \partial_z u'.
\end{equation}
In particular thanks to \eqref{eqn:HodographDerivatives} and \eqref{eqn:HodographJacobian}, locally we have the following transformation of the energy functional
\[
J \longrightarrow \int_{B_1^+} \frac{1+ \left(\partial_s u' \right)^2 + \left(\partial_z u' \right)^2}{\partial_z u'}.
\]
Hence, we have that $\tilde{u}$ is a local minimizer, with respect to smooth variations $\varphi \in C^{1}_c(B_1)$ that guarantee $\tilde{u} + \varphi \ge \tilde{v}$ on $\{ z=0 \}$, of the functional
\[
\tilde{J} \coloneqq \frac{1}{2} \int_{B_1^+} \frac{\vert \nabla \tilde{u} \vert^2}{1 + \partial_z \tilde{u}} = \int_{B_1^+} F\left( \nabla \tilde{u} \right),
\]
where we have introduced the lagrangian
\begin{equation}\label{eqn:hodographLagrangian}
F(x, y) \coloneqq \frac{x^2 + y^2}{2(1+y)}.
\end{equation}
Proceeding in the same way for $\tilde{v}$, a first variation argument shows that the couple $\tilde{u},\tilde{v}$ solves the system \eqref{e:thin-two-membrane-system}. Since $\mathcal{F}$ is analytic and satisfies \eqref{e:definition-mathcal-F}, we conclude from \cref{t:thin-two-membrane-general} that its branching points are isolated. This proves \cref{thm:onePhaseAnalyticObstacle} (a).

\subsubsection{Structure of the free-boundary} In this section we prove \cref{thm:onePhaseAnalyticObstacle} (b). Let the function $w:B_1^+ \to \R$ be defined as
\[
w \coloneqq \tilde{u} - \tilde{v}.
\]
To begin with, we prove that in dimension $d=2$ the function $w$ cannot decay with infinite order.
\begin{lemma}\label{lemma:onePhaseObstDiffLInftyDecayFinite}
There exist a real number $\gamma > 0$ and a sequence of radii $r_n \to 0^+$ such that
\begin{equation}\label{eqn:onePhaseObstLInftyFiniteOrder}
\lim_{n \to +\infty} \left\| \frac{w\left( r_n x \right)}{r_n^{\gamma}} \right\|_{L^{\infty}\left( B_r^+ \right)} \in (0, +\infty].
\end{equation}
\end{lemma}
\begin{proof}
    As noticed in \cref{rmk:Thin2membraneDiffProperty} the difference function $w$ is a a variational solution of the thin obstacle problem \eqref{eqn:variationalThinObstDefProof}. Hence, let $f:\R^2 \to \R^2$ be the quasiconformal map from \cref{t:thin-obstacle-general} and the function $h \coloneqq w \circ f$ be a solution of the harmonic thin obstacle problem. In particular, $h$ cannot decay with infinite order, so that there exists a real number $\beta > 0$ such that
    \begin{equation}\label{eqn:thinObstacleLInftyFiniteOrder}
    \lim_{n \to +\infty} \left\| \frac{h\left( r_n x \right)}{r_n^{\beta}} \right\|_{L^{\infty}\left( B_r^+ \right)} \in (0, +\infty],
    \end{equation}
    for all vanishing sequence of radii $r_n$.
    
    In order to conclude the proof we recall the following properties of quasiconformal maps, for which we refer for instance to \cite[Chapter 3]{AstalaIwaniecMartin:PDEsQuasiconformal}:
    \begin{itemize}
        \item quasiconformal maps are locally invertible, with the inverse still being quasiconformal;

        \item quasiconformal maps are locally H\"older continuous.
    \end{itemize}
    The two properties above imply that
    \begin{equation}\label{eqn:onePhaseObstDiffQConformalDecay}
        \vert f^{-1}(z) \vert \ge c \vert z \vert^{\delta}
    \end{equation}
    for some constants $c>0$ and $0 < \delta \le 1$.

    Combining \eqref{eqn:thinObstacleLInftyFiniteOrder} with \eqref{eqn:onePhaseObstDiffQConformalDecay} gives \eqref{eqn:onePhaseObstLInftyFiniteOrder}.
\end{proof}
Once we have excluded the possibility of having infinite order, we have to show that the function $w$ has indeed a finite order. To this aim we employ a frequency function approach. Even though the branching points are isolated, since the function $w$ is a variational solution of \eqref{eqn:variationalThinObstDefProof} and the coefficients of the matrix $M$ are a priori not better then $C^{0, 1/2}$ regular, we choose not to proceed with a direct differentiation of the frequancy formula. Instead, we use an alternative approach of Garofalo and Petrosyan \cite{GarofaloPetrosyan2009:NewMonotonicityThinObstacle} based on the (almost-)monotonicity of a family of Weiss' monotonicity formulas, which seems more direct due to its variational nature and is much less affected from the regularity of the coefficients then the differentiation of Almgren's frequency formula. The price to pay is that one can only prove the (almost-)monotonicity for the truncated Almgren's formula, thus not excluding the presence of points with infinite frequency. This, however, is not really a problem for our analysis, since the points with infinite frequency have already been excluded. Notice that, in order to use this approach, it is important that $w$ is a minimizer as pointed out in \cref{rmk:Thin2membraneDiffProperty}, and not just a variational solution in the sense of \eqref{e:variational-inequality-thin}.

Since in the theory of thin obstacle problems with $C^{0, \alpha}$ coefficients these tools were developed in the recent work \cite{JeonPetrosyanSmitvega2020:AlmostMinimizersThinObstacleHolder}, we only limit ourselves to recall the main results without proof.

Let $k_0 \ge 2$ be e fixed real number and for all $0 < k < k_0$ define the Weiss monotonicity formula
\[
W_k^0(w; r) \coloneqq \frac{1}{r^{d + 2k - 2}} \left( \int_{B_r^+} \vert \nabla w \vert^2 - \frac{k}{r} \int_{ \partial B_{r} \cap \{ x_2>0 \} } w^2 \right)
\]
and the following Weiss'-type monotonocity formula
\[
W_k(w; r) \coloneqq \frac{e^{a_kr^{1/2}}}{r^{d + 2k - 2}} \left( 
\int_{B_r^+} \vert \nabla w \vert^2 - k \frac{1 - b r^{1/2}}{r} \int_{ \partial B_{r} \cap \{ x_2>0 \} } w^2 \right),
\]
where the real numbers $a_k$ and $b$ are defined as
\[
a_k \coloneqq 2 \left(d + 2k - 2\right) \quad \text{and} \quad b \coloneqq 2 \left(d + 2k_0\right).
\]
%
    %
    %
%
Then, by \cite[Theorem 7.1]{JeonPetrosyanSmitvega2020:AlmostMinimizersThinObstacleHolder}, we have that there exists a dimensional constant $r_0>0$ such that
    \begin{equation}\label{e:twoPhaseDiffWeissMonotone}
    \frac{d}{dr} W_k(w; r) \ge 0\quad\text{for all}\quad0 < r < r_0.
    \end{equation}
Now, if we define the height function $H(r)$ and the energy function D(r) as
\[
H(w; r) \coloneqq \frac{1}{r^{d-1}} \int_{ \partial B_{r} \cap \{ x_2>0 \} } w^2 \qquad \text{and} \qquad D(w; r) \coloneqq \frac{1}{r^{d-2}} \int_{ B_{r} \cap \{ x_2>0 \} } \vert \nabla w \vert^2,
\]
the classical Almgren's frequency function is given by
\[
N(w; r) \coloneqq \frac{D(w; r)}{H(w; r)}.
\]
Considering the following truncated Almgren's-type frequency formula
\[
\widetilde{N}_{k_0}(w; r) \coloneqq \min\left\{ \frac{1}{1-b r^{1/2}} N(w; r), k_0  \right\},
\]
as a consequence of \eqref{e:twoPhaseDiffWeissMonotone} we have that (see e.g. \cite[Section 7]{JeonPetrosyanSmitvega2020:AlmostMinimizersThinObstacleHolder})
%
    %
    \begin{equation}\label{e:twoPhaseDiffFrequencyMonotone}
    \widetilde{N}_{k_0}(w; r_1) \le \widetilde{N}_{k_0}(w; r_2)\qquad\text{for all}\qquad 0<r_1\le r_2<r_0\,.
    \end{equation}
    %
%
If we choose $k_0 > \gamma$, where $\gamma$ is as in \cref{lemma:onePhaseObstDiffLInftyDecayFinite}, \eqref{e:twoPhaseDiffFrequencyMonotone} and the non-degeneracy condition \eqref{eqn:onePhaseObstLInftyFiniteOrder} imply the following. Let the real number $l$ be defined as
\[
l \coloneqq \lim_{r \to 0^+} \widetilde{N}_{k_0}(w; r).
\]
Then, 
\begin{equation}\label{crl:twoPhaseDiffFrequencyLimit}
    1<l \le \gamma.
\end{equation}
In particular, $H(w; r)$ has a polynomial rate of decay equal to $l$, in the sense that
\begin{equation}\label{eqn:polynomialRateDecayFrequencyBlowup}
H(w; r) \le C r^{l} \text{ in } (0, r_0) \quad \text{and} \quad \lim_{r \to 0^+} \frac{r^{\beta}}{H(w; r)} = 0 \text{ for all } \beta \in (l, +\infty)   
\end{equation}
for some constant $C >0$.
\begin{remark}\label{rmk:twoPhaseDiffFrequencyLimit}
    The notion of polynomial rate of decay in \eqref{eqn:polynomialRateDecayFrequencyBlowup} in general does not imply that 
    \[
    \lim_{r \to 0^+} \frac{H(w; r)}{r^l} \in (0, +\infty).
    \]
    Indeed, it does not exclude that the limit vanishes, for instance with a behavior of the type
    \[
    H(w; r) \sim \frac{r^l}{\ln\left(\frac{1}{r} \right)}.
    \]
    This, however, will be achieved later on employing an epiperimetric inequality.
\end{remark}
The main consequence of \eqref{e:twoPhaseDiffFrequencyMonotone} and \eqref{crl:twoPhaseDiffFrequencyLimit} is the existence of non-degenerate $l-$homogeneous blow-ups. More precisely, the following lemma holds (we refer to \cite[Section 8]{JeonPetrosyanSmitvega2020:AlmostMinimizersThinObstacleHolder} for the details of the proof).
\begin{lemma}\label{lemma:twoPhaseDiffFrequencyBlowUp}
    Let the rescalings $w_r$ be defined as
\[
w_r \coloneqq \frac{w(r x)}{H(r)^{1/2}} \quad \text{in } B_1^+, \quad r \in (0, r_0).
\]
Then, for any vanishing sequence $\{r_n\}$, there exists a nontrivial positively l-homogeneous blow-up $w_0:B_1^+ \to \R$ with $w_0 \in H^1(B_1^+)$ (possibly depending on the chosen sequence), solution of the harmonic thin obstacle problem in $B_1^+$ and such that, up to a (non relabeled) subsequence
\begin{align*}
    & w_{r_n} \to w_0 \text{ weakly in } H^1(B_1^+), \\
    & w_{r_n} \to w_0 \text{ strongly in } C^{1, \alpha}(\overline{B_1^+} \cap K),
\end{align*}
for some $0 < \alpha \le 1$ and any set $K \Subset B_1$. In particular, since the origin is a branching point, it follows that
\[
l = 2m- \frac{1}{2} \quad \text{ for some } m \in \N \setminus \{ 0 \}
\]
\end{lemma}
In order to complete the proof of \cref{thm:onePhaseAnalyticObstacle} we have to prove uniqueness for the frequency blow-up and an exact polynomial estimate. This is carried out via an epiperimetric inequality.

To begin with, from \cite{ColomboSpolaorVelichkov2020:DirectEpiperimetricThinObstacle} we recall the epiperimetric inequality for the non-integer frequencies in dimension $d=2$.
\begin{proposition}\label{prop:onePhaseObstacleEpiperimetric2D}
    Let $l = 2m-1/2$ for some $m \in \N \setminus \{ 0 \}$ and $u \in H_1\left( B_1^+ \right)$ be $l-$homogeneous with $w \ge 0$ on $B_1'$. Then, there exists $\eta > 0$ and a function $v \in H_1\left( B_1^+ \right)$ such that:
    \begin{enumerate}
        \item $v = u$ on $\partial B_1 \cap \{ x_2 \ge 0 \}$;
        
        \item $u \ge 0$ on $B_1'$;

        \item $W_l^0(v; 1) \le (1-\eta) W_l^0(u; 1)$.
    \end{enumerate}
\end{proposition}
Following \cite[Section 9]{JeonPetrosyanSmitvega2020:AlmostMinimizersThinObstacleHolder}, as a consequence of the \cref{prop:onePhaseObstacleEpiperimetric2D} we have that there exists constants $C, r_0, \delta, \eta > 0$ such that the following properties hold for all $0 < r < r_0$:
 \begin{enumerate}
        \item $0 \le W_l(w; r) \le C r^{\delta}$;

        \item $H(r) \le C r^{2l}$ and $D(r) \le C r^{2l-1}$;

        \item $H(r) \ge \eta r^{2l}$.
    \end{enumerate}
As a consequence of the above properties, we have the following
\begin{proposition}
Let the rescalings $\widetilde w_r$ be defined as
\[
\widetilde w_r \coloneqq \frac{w(r x)}{r^l} \quad \text{in } B_1^+, \quad r \in (0, r_0).
\]
Then, there exists a unique nontrivial and positively l-homogeneous blow-up $\widetilde w_0:B_1^+ \to \R$ with $\widetilde w_0 \in H^1(B_1^+)$, solution of the harmonic thin obstacle problem in $B_1^+$ and such that
\begin{align*}
    & \widetilde w_{r} \to \widetilde w_0 \text{ weakly in } H^1(B_1^+), \\
    & \widetilde w_{r} \to \widetilde w_0 \text{ strongly in } C^{1, \alpha}(\overline{B_1^+} \cap K),
\end{align*}
for some $0 < \alpha \le 1$ and any set $K \Subset B_1$.
Moreover, the following decay estimate holds:
\[
\frac{1}{r^d}\int_{B_r} \vert w - \widetilde w_0 \vert \le C r^{\delta/2}.
\]
\end{proposition}

\section{Proof of \cref{t:one-sided-two-phase}}
To begin with, we perform two hodograph transformations on $u$ and $v$ separately. Let us proceed with $u$. The computations for $v$ are analogous.

 \medskip
 
 Let $x = (x', x_N)$ and consider the map 
 \[
 \Phi(x) = \left( x', \frac{u(x)}{\sqrt{\Lambda_u}} \right).
 \]
 Since $\partial_N u(0) = 1$, $\Phi$ is locally invertible, and introduce the function $\tilde{u}(y)$ defined via the relation
 \[
 \Phi^{-1}(y) = (y', \tilde{u}(y)).
 \]
 Notice that by construction
 \[
 \tilde{u}(x', u(x)) = x_N.
 \]
 Differentiating this relation, we get that
 \[
    \begin{cases}
        \partial_{y'} \tilde{u} + \partial_{x'} u \, \partial_N \tilde{u} = 0, \\
        \partial_N u \, \partial_N \tilde{u} = 1.
    \end{cases}
 \]
 If we consider the linearized functions $w_{u} = \tilde{u} - x_N$ and $w_{v} = \tilde{v} - x_N$, the one-sided energy functional
 \[
 \int_{B_1} \left(\vert \nabla u \vert^2 + \vert \nabla v \vert^2 \right) + \Lambda_u \vert \{ u>0 \} \vert + \Lambda_v \vert \{ v>0 \} \vert
 \]
 can be rewritten as
 \[
 \Lambda_u \int_{B_1^+} \frac{\vert \nabla w_u \vert^2}{1+\partial_N w_u} + \Lambda_v \int_{B_1^+} \frac{\vert \nabla  w_v \vert^2}{1+\partial_N w_v}.
 \] 
 This implies that the functions $w_u$ and $w_v$ are solutions of a (quasilinear) thin two-membrane problem of the form \eqref{e:thin-two-membrane-system}, with
 \[
 \mathcal{F}(x, y) = \frac{x^2+y^2}{1+y}.
 \]
\cref{t:one-sided-two-phase} (a) is now a direct consequence of \cref{t:thin-two-membrane-general}.

Concerning \cref{t:one-sided-two-phase} (b), let
\[
w = w_v - w_u.
\]
By the quasilinearity, the function $w$ is a solution of a viariable coefficients thin-obstacle problem, with $C^{0, 1/2}$ regular coefficients.
\begin{equation*}
    \begin{cases}
        \diverg\left( M(x) \nabla w \right) = 0 & \text{in } B_1^+ , \\
        w \ge 0 & \text{on } \{ x_N = 0 \} , \\
        M(x) \nabla w \cdot e_N \le 0 & \text{on } \{ x_N = 0 \} , \\
        M(x) \nabla w \cdot e_N = 0 & \text{on } \{ w > 0 \} \cap \{ x_N = 0 \}.
    \end{cases}
\end{equation*}
for some symmetric and uniformly elliptic matrix field $M(x)$. Then, in order to study the rate of decay of the function $w$, one can proceed as in \cref{sec:proofThmonePhaseAnalyticObstacle}.

\appendix

\section{Extension of analytic free boundaries}

Let $u:D\to\R$ be a non-negative Lipschitz continuous function in an open set $D\subset\R^2$. We will say that $u$ is a classical solution of the one-phase Bernoulli problem in $D$ if $\partial\{u>0\}$ is $C^1$-regular in $D$ and if $u$ is $C^1$ in the set $\overline{\{u>0\}}\cap D$ and if
\begin{align*}
\Delta u=0\quad\text{in}\quad \{u>0\}\cap D\ ,\quad|\nabla u|=1\quad\text{on}\quad \partial\{u>0\}\cap D\,.
\end{align*}
It is well-known (see \cite{AltCaffarelli:OnePhaseFreeBd} and \cite{KinderlehrerNirenberg}) that if $u$ is a classical solution, then the free boundary $\partial\{u>0\}\cap D$ is analytic.

\begin{theorem}\label{t:1}
Let $u:D\to\R$ be a classical solution to the one-phase Bernoulli problem in the open set $D\subset\R^2$. Then, the free boundary $\Gamma=\partial\{u>0\}\cap D$ is the union of analytic curves. Moreover, if $(x_0,y_0)\in \Gamma$ is a free boundary point, then the solution $u$ can be extended to a harmonic function $\widetilde u$ in a ball $B_r(x_0,y_0)\subset D$, whose nodal set is precisely the free boundary $\Gamma$, that is, $\{\widetilde u=0\}\cap B_r(x_0,y_0)=\Gamma\cap B_r(x_0,y_0)$.
\end{theorem}	

In this section, we prove that also the converse is true. Our main result is the following. 

\begin{theorem}\label{t:2}
Let $f:(-R,R)\to\R$ be an analytic function on the interval $\mathcal I_R=(-R,R)$ such that $f(0)=f'(0)=0$. Then, there are $\rho\in(0,R)$ and a function $u:\mathcal I_\rho\times \mathcal I_\rho\to\R$ which is a classical solution of the one-phase Bernoulli problem in $\mathcal I_\rho\times \mathcal I_\rho$ with free boundary which is given precisely by the graph of $f$ over $\mathcal I_\rho$.
\end{theorem}	

In \cref{s:proof-t1} we will use the conformal hodograph transform from \cite{DePhilippisSpolaorVeluchkov2021:QuasiConformal2D} to give another proof of \cref{t:1}, which follows step-by-step the procedure from \cite{DePhilippisSpolaorVeluchkov2021:QuasiConformal2D}; the only difference lies in the absence of inclusion constraints, which simplifies the analysis and hopefully makes the proof easier to follow. In \cref{s:proof-t2} we invert this procedure in order to construct analytic solutions with prescribed free boundary, thus proving \cref{t:2}.

\subsection{Analyticity of the classical solutions to the one-phase problem}\label{s:proof-t1}

Let $u$ be a classical solution of the one-phase problem and let $(0,0)\in\partial\{u>0\}$.\\ We can suppose that there are $\rho>0$ and a $C^1$ function 
$$f:\mathcal I_\rho\to\R\ ,\quad \text{where $\mathcal I_\rho$ is the interval $(-\rho,\rho)$\,},$$
which describes the free boundary $\partial \Omega_u$ in the square $\mathcal I_\rho\times\mathcal I_\rho$. We assume that
$$f(0)=f'(0)=0\ ,\quad -\rho< f< \rho \quad\text{in}\quad\mathcal I_\rho,$$
and that 
$$\Omega_\rho:=\Big\{(x,y)\in \mathcal I_\rho\times\mathcal I_\rho\ :\ y>f(x)\Big\}=\{u>0\}\cap (\mathcal I_\rho\times\mathcal I_\rho).$$
Thus, the free boundary in $\mathcal I_\rho\times\mathcal I_\rho$ is the graph of $f$:
$$\Gamma_\rho:=\Big\{(x,y)\in \mathcal I_\rho\times\mathcal I_\rho\ :\ y=f(x)\Big\}=\partial\{u>0\}\cap (\mathcal I_\rho\times\mathcal I_\rho).$$
In particular, since 
\begin{equation}\label{e:identities-along-Gamma}
u(t,f(t))=0\quad\text{and}\quad \partial_x u(t,f(t))+f'(t)\partial_y u(t,f(t))=0\quad\text{for every}\quad t\in\mathcal I_\rho\,,
\end{equation}
we have that $\partial_{x}u(0,0)=0$ and so the free boundary condition gives that 
	$$ \partial_{x}u(0,0)=0\qquad\text{and}\qquad\partial_y u(0,0)= 1\,.$$

\subsubsection{The harmonic conjugate of $u$} We define the function 	
$$U: \Omega_\rho\cup \Gamma_\rho\to\R\,,$$
as the line integral $\,\displaystyle\int_\sigma \alpha\,$ of the closed $1$-form 
$$\alpha:=\partial_y u(x,y)\,dx-\partial_xu(x,y)\,dy\,,$$
over any curve
$$\sigma:[0,1]\to\Omega_\rho\cup \Gamma_\rho$$
connecting the origin $(0,0)$ to $(x,y)$. Notice that, since $\Omega_\rho$ is simply connected, the value $U(x,y)$ depends only on the end-point $(x,y)$ and not on the choice of the curve $\sigma$. Moreover, by construction $U\in C^{1,\alpha}(\Omega_\rho\cup \Gamma_\rho)$ and 
\begin{equation}\label{e:identities-U}
U(0,0)=0,\quad \partial_xU=\partial_yu\quad\text{and}\quad \partial_yU=-\partial_xu\quad\text{in}\quad \Omega_\rho.
\end{equation}
\begin{remark}[Definition of $\eta$]\label{rem:eta}
We notice that the value of $U$ at a free boundary point 
$$(x,f(x))\in \Gamma_\rho\ ,$$ 
is precisely the length of the curve  
	$$\sigma:[0,x]\to\R^2,\quad \sigma(t)=(t,f(t)).$$
In fact, by the definition of $U$, we have
	\begin{align*}
	U(x,f(x)):=\int_\sigma\alpha&=\int_0^x\Big(\partial_yu(t,f(t))-\partial_xu(t,f(t))f'(t)\Big)\,dt\\
	&=\int_0^x|\nabla u|(t,f(t))\sqrt{1+f'(t)^2}\,dt=\int_0^x\sqrt{1+f'(t)^2}\,dt\,.
	\end{align*}
	In what follows, we will use the notation 
	$$\eta(x):=U(x,f(x))=\int_0^x\sqrt{1+f'(t)^2}\,dt.$$
\end{remark}
	
\subsubsection{The conformal hodograph transform}	
We define the change of coordinates 
	$$x'=U(x,y)\,,\quad y'=u(x,y)\,,$$
	given by the $C^{1,\alpha}$-regular map 
	$$T:\Omega_\rho\cup\Gamma_\rho\to\R^2\cap\{y'\ge 0\}\,,\qquad T(x,y)=(x',y')\,.$$
	By the definition of $U$ and the fact that  $\partial_ y u(0,0)\ge 1$, we have that for $\rho$ small enough, the map $T$ is invertible. 
	\begin{remark}
	For $\rho$ small enough, the set $\ T\big(\Omega_\rho\cup\Gamma_\rho\big)\ $ is a relatively open subset the upper half-plane $\R^2\cap\{y'\ge 0\}$ and we have the identities: 
	$$T(\Omega_\rho)=T\big(\Omega_\rho\cup\Gamma_\rho\big)\cap\{y'>0\}\qquad\text{and}\qquad T(\Gamma_\rho)=T\big(\Omega_\rho\cup\Gamma_\rho\big)\cap\{y'=0\}.$$ 
	\end{remark}

Let $S$ be the inverse of $T$:
		$$S(x',y')=\big(V(x',y'),v(x',y')\big)\ ,\qquad S:T\big(\Omega_\rho\cup\Gamma_\rho\big)\to \Omega_\rho\cup\Gamma_\rho\ .$$
 We will sometimes write it in coordinates as
	$$S(x',y')=(x,y)\,,\quad x=V(x',y')\,,\quad y=v(x',y')\,.$$
\begin{remark}\label{rem:graph}
We notice that the geometric information about the free boundary $\Gamma_\rho$ is carried by the second component of $S$.
In fact, in the new coordinates $(x',y')$, the free boundary $\Gamma_\rho$ is the graph of the function 
$$x'\mapsto v(x',0).$$ 
In order to show this, let us fix a point $(x,y)\in\Omega_\rho\cup \Gamma_\rho$. Notice that  
$$(x,y)\in\Gamma_\rho\ \Leftrightarrow\ y=f(x)\ \Leftrightarrow\ u(x,y)=0\,,$$
and that we have the relation
\begin{equation}\label{e:inverse-identity-y}
y=v\big(U(x,y),u(x,y)\big)\qquad\text{for every}\qquad (x,y)\in \Omega_\rho\cup\Gamma_\rho\ .
\end{equation}
Thus, 
$$(x,y)\in\Gamma_\rho\ \Leftrightarrow\ u(x,y)=0\ \Leftrightarrow\ y=v(U(x,y),0).$$
Now, since $S$ is a one-to-one correspondance between the sets $T(\Gamma_\rho)$ and $\Gamma_\rho$, we get that 
$$y=v(U(x,y),0)\ \Leftrightarrow\ y=v(U(x,f(x)),0)\ \Leftrightarrow\ y=v(x',0),$$
where we used that 
$$x'=U(x,f(x))=\eta(x).$$
		As a consequence, we get the identity
		$$f(x)=v(\eta(x),0)\qquad\text{for every}\qquad x\in\mathcal I_\rho\,.$$
\end{remark}
\begin{definition}
We will call $v$ conformal hodograph transform of $u$.	
\end{definition}		
	\begin{lemma}\label{lem:con1ph}
		Let $T=(U,u)$ and $S=(V,v)$ be as above. Then, the function
		$$v:T(\Omega_\rho)\cup T(\Gamma_\rho)\to\R,$$
		is  $C^{1}$-regular in $T(\Omega_\rho)\cup T(\Gamma_\rho)$, $C^\infty$-regular in the open set $T(\Omega_\rho)$, and is such that
\begin{equation}\label{e:equations-v}
\Delta v=0\quad\text{in}\quad T(\Omega_\rho)\ ,\qquad |\nabla v|= 1\quad\text{on}\quad T(\Gamma_\rho).
\end{equation}
Moreover, for every $x\in \Gamma_\rho$, we have the identities
\begin{equation}\label{eq:etauu}
f'(x)= \frac{\partial_{x'}v(\eta(x),0)}{\partial_{y'}v(\eta(x),0)}
\qquad \text{and}\qquad 
\eta'(x)=\frac{1}{\partial_{y'}v(\eta(x),0)}.
\end{equation}
	\end{lemma}	
	\begin{proof}  We first notice that $v$ is harmonic in $T(\Omega_\rho)$ since it is the second component of a conformal map. Moreover, using \eqref{e:inverse-identity-y} and taking the derivatives with respect to $x$ and $y$, we get
		\begin{align*}
		\partial_{x'} v\big(U(x,y),u(x,y)\big)\partial_xU(x,y)+	\partial_{y'} v\big(U(x,y),u(x,y)\big)\partial_xu(x,y)=0,\\
	\partial_{x'} v\big(U(x,y),u(x,y)\big)\partial_yU(x,y)+	\partial_{y'} v\big(U(x,y),u(x,y)\big)\partial_yu(x,y)=1.
	\end{align*}
Using the identities \eqref{e:identities-U} for the partial derivatives of $U$, we obtain the system 
	 \begin{align*}
	 \partial_{x'} v(x',y')\,\partial_yu(x,y)+	\partial_{y'} v(x',y')\,\partial_xu(x,y)=0,\\
	 -\partial_{x'} v(x',y')\,\partial_xu(x,y)+	\partial_{y'} v(x',y')\,\partial_yu(x,y)=1,
	 \end{align*}
which leads to
	 \begin{equation}\label{e:boh}
	 \partial_{y'}v(x',y')=\frac{\partial_yu(x,y)}{|\nabla u|^2(x,y)}
	 \qquad \text{and}\qquad
	 \partial_{x'}v(x',y')=-\frac{\partial_xu(x,y)}{|\nabla u|^2(x,y)}\,.
	 \end{equation}
In particular, we have that, for every $(x,y)\in \Omega_\rho\cup\Gamma_\rho$,
	 \begin{equation}\label{e:gradv}
	 |\nabla u|(x,y)\, |\nabla v|(x',y')=1\,,
	 \end{equation}
	 where $(x',y')=(U(x,y),u(x,y))$. This concludes the proof of \eqref{e:equations-v}.

We next prove \eqref{eq:etauu}. By \eqref{e:identities-along-Gamma}, we have that
	 $$ f'(x)=- \frac{\partial_{x}u(x,f(x))}{\partial_{y}u(x,f(x))}\quad\text{for every}\quad x\in\mathcal I_\rho.$$
	On the other hand, by \eqref{e:boh} and the fact that $x'=\eta(x)$, we get 
	\begin{equation}\label{e:boh}
	\partial_{y'}v(\eta(x),0)=\partial_yu(x,f(x))
	\qquad \text{and}\qquad
	\partial_{x'}v(\eta(x),0)=-\partial_xu(x,f(x))\,.
	\end{equation}
	Thus, we obtain:
$$f'(x)= \frac{\partial_{x'}v(\eta(x),0)}{\partial_{y'}v(\eta(x),0)}\quad\text{for every}\quad x\in\mathcal I_\rho.$$
Using \cref{rem:eta}, we get that for every $x\in\mathcal I_\rho$
$$\eta'(x)=\sqrt{1+f'(t)^2}=\frac{1}{\partial_{y'}v(\eta(x),0)}\ ,$$
which concludes the proof.
	 	\end{proof}

\subsubsection{Analyticity of the free boundary. Proof of \cref{t:1}}
We define the complex gradient 
$$Q:=\partial_{z'} v=\partial_{x'}v-i\partial_{y'}v,$$
where $z'=x'+iy'$. Then $Q$ satisfies
\begin{equation*}
\partial_{\bar z'}Q=0\quad\text{in}\quad T(\Omega_\rho)\ ,\quad
|Q|= 1\quad\text{on}\quad T(\Gamma_\rho)\,.
\end{equation*}
Consider the function
$$P:T(\Omega_\rho)\cup T(\Gamma_\rho)\to\R\ ,\quad P=-i\frac{Q+i}{Q-i}\ .$$
Then, $P$ is holomorphic in $T(\Omega_\rho)$ and since $P$ can be also written as
$$P=\frac{2\,\text{Re}\,Q}{|Q-i|^2}-i\frac{|Q|^2-1}{|Q-i|^2}\,,$$
we get that $P$ satisfies
\begin{equation*}
\partial_{\bar z'}P=0\quad\text{in}\quad T(\Omega_\rho)\ ,\quad
\text{Im}\, P=0\quad\text{on}\quad T(\Gamma_\rho).
\end{equation*}
In particular, this implies that $P$ can be extended to an holomorphic function across $T(\Gamma_\rho)$. As a consequence, also $Q$ is holomorphic across $T(\Gamma_\rho)$ and so, the imaginary part $\partial_{y'}v(x',0)$ is analytic function of $x'$. Using the fact that the function $\eta:(-\rho,\rho)\to\R$ solves the equation 
$$\displaystyle\eta'(x)=\frac{1}{\partial_{y'}v(\eta(x),0)}\ ,\qquad 
\eta(0)=0\,,$$
we obtain that $\eta$ is analytic. As a consequence, also $f$ is analytic. This concludes the proof of \cref{t:1}.\qed

\subsection{Construction of analytic solutions}\label{s:proof-t2}
Let $f:\mathcal I_\rho\to\R$ be an analytic function such that 
$$f(0)=f'(0)=0\quad\text{and}\quad |f'|\le 1\quad\text{on}\quad\mathcal I_\rho\,.$$
In particular, also $f'(x)$ is analytic on $\mathcal I_\rho$. Define the function
$$\eta:\mathcal I_\rho\to\R\ ,\qquad\eta(x):=\int_{0}^x\sqrt{1+(f'(t))^2}\,dt\ .$$
Then, $\eta$ is analytic on $\mathcal I_\rho$, 
$$\eta(0)=0\ ,\quad \eta'(0)=1\quad\text{and}\quad \eta'\ge 1\quad\text{on}\quad \mathcal I_\rho\,.$$
Let $\eta^{-1}:\eta(\mathcal I_\rho)\to\mathcal I_\rho$ be the inverse of $\eta$. We next define the function
$$\beta:\eta(\mathcal I_\rho)\to\R\ ,$$
by the identity
$$\eta'(x)=\frac1{\beta(\eta(x))}.$$
We notice that $\beta$ is anlaytic and that $0<|\beta|\le 1$ on the interval $\eta(\mathcal I_\rho)$. Similarly, let 
$$\alpha:\eta(\mathcal I_\rho)\to\R\ ,$$
be the analytic function defined trough the identity
$$f'(x)=\frac{\alpha(\eta(x))}{\beta(\eta(x))}\quad\text{for every}\quad x\in\mathcal I_\rho.$$
We notice that since by construction 
$$\eta'(x)=\sqrt{1+(f'(x))^2}\,,$$
we get that 
$$\frac1{\beta(\eta(x))}=\sqrt{1+\frac{\alpha(\eta(x))^2}{\beta(\eta(x))^2}}\,$$
which leads to 
$$\alpha(x')^2+\beta(x')^2=1\quad\text{for every}\quad x'\in\eta(\mathcal I_\rho)\,.$$
We define the function 
$$Q:\eta(\mathcal I_\rho)\to\mathbb C\ ,\qquad Q(x'):=\alpha(x')-i\beta(x')\,.$$
Then, $Q$ is analytic (in the real variable $x'$) on $\mathcal I_\rho$ and 
$$Q(0,0)=-i	\qquad\text{and}\qquad|Q|=1\quad\text{on}\quad \eta(\mathcal I_\rho).$$
As a consequence, in a small ball $B_\delta$, the function $Q$ can be extended to a holomorphic function in the complex variable $z'=x'+iy'$. We will use the notation 
$$Q(z')=Q(x',y')=\alpha(x',y')-i\beta(x',y').$$
Notice that since $Q$ is holomorphic on $B_\delta$, we have 
$$\partial_{\bar z'}Q=0,$$
or, equivalently, for every $(x',y')\in B_\delta$,
$$\partial_{x'}\alpha(x',y')=-\partial_{y'}\beta(x',y')\quad\text{and}\quad \partial_{y'}\alpha(x',y')=\partial_{x'}\beta(x',y')\,.$$
In particular, this means that the $1$-forms 
$$\alpha(x',y')\,dx'+\beta(x',y')\,dy'\qquad\text{and}\qquad -\beta(x',y')\,dx'+\alpha(x',y')\,dy'$$
are both closed in $B_\delta$. Then, we can find functions
$$V:B_\delta\to\R\qquad\text{and}\qquad v:B_\delta\to\R$$
such that 
$$V(0,0)=0\ ,\quad \partial_{x'}V(x',y')=-\beta(x',y')\ ,\quad \partial_{y'}V(x',y')=\alpha(x',y')\,;$$
$$v(0,0)=0\ ,\quad \partial_{x'}v(x',y')=\alpha(x',y')\ ,\quad \partial_{y'}v(x',y')=\beta(x',y')\,.$$
First, we notice that $v$ solves the following problem in $B_\delta$ :
$$\Delta v=0\quad\text{in}\quad B_\delta\ ,\qquad |\nabla v|=1\quad\text{on}\quad B_\delta\cap\{y'=0\}.$$
Second, the map 
$$S:B_\delta\to \R^2\ ,\qquad S(x',y'):=\big(V(x,y),v(x,y)\big)\ ,$$
is a conformal change of coordinates and is such that 
$$S(0,0)=(0,0)\quad\text{and}\quad\nabla S(0,0)=\begin{pmatrix}-1&0\\0&1\end{pmatrix}\ ,$$
so it is invertible in a neighborhood $B_\eps$ of $(0,0)$. Let $T$ be its inverse and let $U$ and $u$ be the components of $T$
$$T:S(B_\eps)\to \R^2\ ,\qquad T(x,y):=\big(U(x,y),u(x,y)\big)\ .$$
We define the sets $\Omega\subset\R^2$ and $\Gamma\subset\R^2$ as 
$$\Omega:=S\big(B_\eps\cap\{y'>0\}\big)\qquad\text{and}\qquad \Gamma:=S\big(B_\eps\cap\{y'=0\}\big).$$
Since $S$ is a homeomorphism, we have that 
$$\partial \Omega=\Gamma\quad\text{in the open set}\quad S\big(B_\eps).$$
Moreover, by the definion of $u$, we have that 
$$u>0\quad\text{in}\quad\Omega\qquad\text{and}\qquad u=0\quad\text{on}\quad\Gamma\,.$$
Notice, that $T$ is holomorphic (since it's the inverse of a holomorphic function) and so, $u$ is harmonic in the whole open set $S(B_\eps)$. Moreover, reasoning as in \cref{lem:con1ph}, we get
$$|\nabla u|=1\quad\text{on}\quad\Gamma\,.$$
In conclusion, if we take the positive part $u_{\text{\tiny+}}$ of the function $u$ in $S(B_\eps)$, then we have $u_{\text{\tiny+}}$ is a solution of the one-phase free boundary problem 
$$\Delta u_{\text{\tiny+}}=0\quad\text{in}\quad\Omega=\{u_{\text{\tiny+}}>0\}\cap S(B_\eps)\ ,\qquad |\nabla u_{\text{\tiny+}}|=1\quad\text{on}\quad\Gamma=\partial\{u_{\text{\tiny+}}>0\}\cap S(B_\eps)\,.$$
It only remains to show that $\Gamma$ is the graph of $f$. By the definiton of $\Gamma$ we know that (up to taking $\eps>0$ small enough) $\Gamma$ is the graph of a certain function $g$ with $g(0)=g'(0)=0$. Now, by \cref{rem:graph} and \cref{lem:con1ph}, we get that 
$$\mu'(x)=\frac{1}{\partial_{y'}v(\mu(x),0))}=\frac{1}{\beta(\mu(x),0))}\quad\text{where}\quad \mu(x):=\int_0^x\sqrt{1+(g'(t))^2}\,dt\,.$$
Now, by the Cauchy-Lipschitz theory for the ODE 
$$\mu'(x)=\frac{1}{\beta(\mu(x),0))}\ ,\quad \mu(0)=0\ ,$$
we get that $\mu\equiv \eta$. Finally, by the formula 
$$f'(x)=\frac{\alpha(\mu(x))}{\beta(\mu(x))}=g(x)\,,$$
we obtain that $f\equiv g$, which concludes the proof of \cref{t:2}.


\bibliographystyle{plain}
\bibliography{FreeBoundary_bib.bib}


\appendix

\medskip
\small
\begin{flushright}
\noindent 
\verb"lorenzo.ferreri@sns.it"\\
Classe di Scienze, Scuola Normale Superiore\\ 
Piazza dei Cavalieri 7, 56126 Pisa (Italy)
\end{flushright}

\begin{flushright}
\noindent 
\verb"lspolaor@ucsd.edu"\\
Department of Mathematics, UC San Diego,\\
 AP\&M, La Jolla, California, 92093, USA
\end{flushright}

\begin{flushright}
\noindent 
\verb"bozhidar.velichkov@unipi.it"\\
Dipartimento di Matematica, Università di Pisa\\ 
Largo Bruno Pontecorvo 5, 56127 Pisa (Italy)
\end{flushright}

\end{document}